\documentclass[english,12pt,reqno,twoside]{amsart}

\usepackage{amssymb,amsmath,amsthm,soul,color}
\usepackage{graphicx}
\usepackage[noadjust]{cite}
\usepackage{t1enc}
\usepackage[utf8]{inputenc}
\usepackage{braket}
\usepackage[T1]{fontenc}
\usepackage{amsfonts}
\usepackage{a4,indentfirst,latexsym,graphics}
\usepackage[english]{babel}
\usepackage{mathtools}

\usepackage{fixmath}

\usepackage[left=2.1cm,right=2.1cm,top=2cm,bottom=2cm]{geometry}

\usepackage{url}

\usepackage{mathrsfs}

\usepackage{enumitem}
\usepackage[colorlinks=true,urlcolor=blue,
citecolor=red,linkcolor=blue,linktocpage,pdfpagelabels,
bookmarksnumbered,bookmarksopen]{hyperref}
\usepackage[hyperpageref]{backref}
\usepackage[colorinlistoftodos]{todonotes}

\usepackage[normalem]{ulem}

\usepackage{lmodern}

\newtheorem{thm}{Theorem}[section]
\newtheorem{cor}[thm]{Corollary}
\newtheorem{lem}[thm]{Lemma}
\newtheorem{prop}[thm]{Proposition}

\theoremstyle{definition}

\theoremstyle{remark}
\newtheorem{rem}[thm]{Remark}

\numberwithin{equation}{section}

\DeclareMathSymbol{\C}{\mathalpha}{AMSb}{"43}

\newcommand{\bt}{\begin{thm}}
\newcommand{\et}{\end{thm}}
\newcommand{\beq}{\begin{equation}}
\newcommand{\eeq}{\end{equation}}

\newcommand{\n}{\nabla}

\newcommand{\N}{\mathbb{N}}
\newcommand{\R}{{\mathbb{R}}}
\newcommand{\RT}{{\mathbb{R}^3}}

\newcommand{\E}{{\mathcal{E}}}

\newcommand{\K}{{\mathcal{K}}}
\newcommand{\A}{{\mathcal{A}}}
\renewcommand{\P}{{\mathcal{P}}}

\newcommand{\M}{{\mathcal{M}}}
\newcommand{\irt}{\int_{\mathbb{R}^3}}

\newcommand{\weakto}{\rightharpoonup}

\newcommand{\bsub}{\begin{subequations}}
\newcommand{\esub}{\end{subequations}$\!$}

\begin{document}
	\title[Schr\"{o}dinger-Bopp-Podolsky system]
{On a zero mass Schr\"odinger-Bopp-Podolsky system: ground states,  nonexistence results and asymptotic behaviour
}

\author[A. Pomponio]{Alessio Pomponio}
\address{A. Pomponio
\newline\indent Dipartimento di Meccanica, Matematica e Management,\newline \indent
	Politecnico di Bari
	\newline\indent
	Via Orabona 4,  70125  Bari, Italy}
\email{alessio.pomponio@poliba.it}

\author[L. Yang]{Lianfeng Yang}
\address{L. Yang
\newline \indent $^1$ School of Mathematics and Information Science, \newline \indent Guangxi University
\newline \indent Nanning, Guangxi, P. R. China
\newline\indent $^2$ Dipartimento di Meccanica, Matematica e Management,
\newline \indent Politecnico di Bari
\newline\indent	Via Orabona 4,  70125  Bari, Italy
}
\email{yanglianfeng2021@163.com}

    \begin{abstract}
		In this paper, we consider the following {\em zero mass} Schr\"{o}dinger-Bopp-Podolsky system
		\[
			\begin{cases}
				-\Delta u  +q^2\phi u=|u|^{p-2}u,\\
				-\Delta \phi+a^2\Delta^2\phi=4\pi u^2,
			\end{cases}
		 \text{ in } \RT,	
		 \]
		where $a>0$ and $q\ne 0$. 
We complete the study initiated in \cite{CDPSY25}, which relied on a perturbation argument to establish the existence of weak solutions. Here, in contrast, our approach, based on the Mountain Pass Theorem and the splitting lemma, directly yields a ground state solution for $p \in (4,6)$.
		Moreover, by deriving a Pohozaev identity, we further obtain some nonexistence results for suitable $p$. Finally, based on the minimax characterization, we also analyse, in the radial case, the asymptotic behaviour of the solutions obtained as $a\to 0$, thereby establishing a link with the zero mass Schr\"odinger-Poisson system.
		
\end{abstract}

	\maketitle

\section{Introduction}   
    The Bopp-Podolsky theory, developed independently by Bopp \cite{B40} and Podolsky \cite{P42}, is a second-order gauge theory for the electromagnetic field that addresses   the so-called {\sl infinity problem} associated with a point charge in the classical Maxwell theory. Specifically, in Maxwell theory,  by the well-known Gauss law (or Poisson equation), the electrostatic potential $\phi$, for a given charge distribution with density $\rho$, satisfies the equation
	\begin{equation}\label{eq3.6}
		-\Delta \phi=\rho \quad\text{ in } \R^3.
	\end{equation}
	If we let $\rho=4\pi \delta_{x_0}$ with $x_0\in \R^3$, then the fundamental solution of equation \eqref{eq3.6} is $\mathcal{C}(x-x_0)$ with $\mathcal{C}(x):=1/|x|$, and therefore the energy of the corresponding electrostatic field is not finite since
	$$\int_{\R^3}|\nabla \mathcal{C}|^2dx=+\infty.$$
	On the other hand, in the Bopp-Podolsky theory, equation \eqref{eq3.6} is replaced by
	\begin{equation}\label{eq2.5}
		-\Delta \phi+a^2\Delta^2\phi=\rho \quad\text{ in } \R^3,
	\end{equation}
where $a>0$ can be interpreted as a cut-off distance or can be linked to an effective radius for the electron.
	Taking $\rho=4\pi \delta_{x_0}$ again, the explicit solution of equation \eqref{eq2.5} is $\mathcal{K}_a(x-x_0)$, where
    \begin{equation*}\label{K}
        \mathcal{K}_a(x):=\frac{1-e^{-|x|/a}}{|x|},
    \end{equation*}
    and, since
	$$\int_{\R^3}|\nabla \mathcal{K}_a|^2dx+a^2\int_{\R^3}|\Delta \mathcal{K}_a|^2dx<+\infty,$$
the energy of the electrostatic field generated by a point charge is finite. Both Bopp and Podolsky were influenced by the work of Yukawa \cite{yukawa} (see \cite[p. 345]{B40} and \cite[\S 4]{P42}) and indeed the {\em Bopp-Podolsky potential} $\K_a$ involves the Coulomb and the Yukawa potentials (see for more details \cite[Proposition 2.1]{CDPSY25}).

Recently, the authors in \cite{CKP19} considered the Schr\"odinger spectrum of a hydrogen atom 
when the conventional Coulomb pair energy between a point electron and a point proton 
is replaced by its Bopp-Land\'e-Thomas-Podolsky (BLTP) modification. 
They showed that, in order to match experimental Lyman-$\alpha$ transition data, 
the parameter $a$ of the potential $\K_a$ must satisfy
$a \lesssim 10^{-18}~\mathrm{m}$,
which is consistent with existing bounds on the electron radius. 

Moreover, in \cite[Section 2]{DS19}, d'Avenia and  Siciliano coupled a Schr\"{o}dinger field $\psi=\psi(t,x)$ with its electromagnetic field in the Bopp-Podolsky
theory, and, in particular, in the electrostatic case for the standing waves $\psi(t,x)=e^{i\omega t}u(x)$, they first derived and investigated the  following system
	\begin{equation}\label{eq2.6}
		\begin{cases}
			-\Delta u +\omega u +q^2\phi u=|u|^{p-2}u , \\
			-\Delta \phi+a^2\Delta^2\phi=4\pi u^2,
		\end{cases}
        \text{ in } \R^3.
	\end{equation}
    From a physical point of view, in \eqref{eq2.6}, $a>0$, the parameter of the potential $\K_a$, has dimension of the inverse of mass; $q\neq0$ denotes the coupling constant of the interaction between the particle and its electromagnetic field; $\omega\in\R$ is the frequency of the standing wave $\psi(t,x)=e^{i\omega t}u(x)$;  $\phi$ is the electrostatic potential. 
    Since then, a lot of works related to system  \eqref{eq2.6} have been carried out. For instance, we refer the reader to the interesting results obtained in \cite{DPRS23,FS21,RSS24} for the problem with different constraints,  in \cite{CT20,LPT20,SDS24} for critical cases, in \cite{WCL22,Z24} concerning sign-changing solutions, in \cite{DG22,H19,H20} for nonlinear Schr\"odinger-Bopp-Podolsky-Proca system on  manifolds,  \cite{CLRT22,G23,S20} in further different contexts (see also references therein).

In the case $a=0$,  system \eqref{eq2.6} becomes the well-known Schr\"odinger-Poisson system
	\begin{equation*}\label{eq2.60}
		\begin{cases}
			-\Delta u +\omega u +q^2\phi u=|u|^{p-2}u , \\
			-\Delta \phi=4\pi u^2,
		\end{cases}
        \text{ in } \R^3.
	\end{equation*}
In particular, the {\em zero mass} Schr\"odinger-Poisson system,  namely with $\omega=0$, 
	\begin{equation}\label{SP}\tag{$\P_0$}
		\begin{cases}
			-\Delta u +q^2\phi u=|u|^{p-2}u , \\
			-\Delta \phi=4\pi u^2,
		\end{cases}
        \text{ in } \RT,
	\end{equation}
has also been extensively studied in \cite{IR12,R10}. First,  for $p\in (18/7,3)$, Ruiz in \cite{R10} obtained the existence of a solution to the system \eqref{SP} as a minimizer of the associated energy functional 
\begin{equation*}\label{eq7.16}
	I_0(u) := \frac{1}{2}\|\nabla u\|_2^2 + \frac{q^2}{4}\irt \irt \frac{u^2(x) u^2(y)}{|x-y|}dxdy - \frac{1}{p}\|u\|^p_p,
\end{equation*}
in the  functional space
\begin{equation*}
	E_r:=\left\{u\in D^{1,2}(\RT):u\text{ is radial and }\irt \irt \frac{u^2(x) u^2(y)}{|x-y|}dxdy  <+\infty\right\},
\end{equation*} 
    where    
    $D^{1,2}(\R^3):=\{u\in L^6(\R^3):|\nabla u|\in L^2(\R^3)\}$ endowed with the usual norm $ \|\n \cdot\|_2$.
In \cite{IR12}, instead, the authors proved the existence of a  ground state  and of infinitely many radial bound states for $p\in (3,6)$. For related problems, see also \cite{GL24,MMV16}.

 Motivated by these results, the authors in \cite{CDPSY25}  studied the existence of weak solutions to the zero mass Schr\"odinger-Bopp-Podolsky system, given by:
    \begin{equation}\label{eq1.1}\tag{$\P_a$}
    	\begin{cases}
    		-\Delta u +q^2\phi u=|u|^{p-2}u,  \\
    		-\Delta \phi+a^2\Delta^2\phi=4\pi u^2,
    	\end{cases}
    	\text{ in }  \R^3,
    \end{equation}
    where $a>0$, $q\neq 0$, and $p\in (3,6)$. To this end,  {\em Bopp-Podolsky energy} was introduced as 
    \begin{equation*}\label{eq:functional}
    	V(f,g) :=  \int_{\R^3} \int_{\R^3} \K_a(x-y) f(x) g(y)  dxdy
    \end{equation*}
    for measurable functions $f,g\colon \R^3\to \R$, and the space $\E$ was defined as the set of functions in $D^{1,2}(\R^3)$ with finite  Bopp-Podolsky energy:
    \begin{equation*}\label{E}
    	\E:=\left\{u \in D^{1,2}(\RT) : V(u^2,u^2)<+\infty\right\}
    \end{equation*}
    equipped with the norm
    $$\|u\|_{\E}:=\left(\|\nabla u\|^2_2+V(u^2,u^2)^\frac{1}{2}\right)^\frac{1}{2}.$$
    Meanwhile, let $\E_r$ denote the subspace of radial functions in $\E$. 
In addition, set
\begin{equation*}\label{eq:A}
	\mathcal{A}:=\Big \{\varphi\in  D^{1,2}(\R^3):\Delta\varphi \in L^2(\R^3)\Big\}
\end{equation*}
equipped with the scalar product
$$\langle\varphi,\psi\rangle_{\mathcal{A}}:=\int _{\R^3}\nabla\varphi\cdot\nabla\psi dx+ a^2\int _{\R^3}\Delta\varphi\Delta\psi dx.$$
We denote by $\|\cdot\|_\A$ the associated norm.
From \cite{DS19}, we know that $C_c^\infty(\R^3)$  is dense in  the Hilbert space $\A$ and
\begin{equation}\label{eq7.22}
	\mathcal{A}\hookrightarrow L^\tau(\R^3), \quad\text{ for }\tau\in [6,+\infty].
\end{equation}
We also write $\A_r$ for the subspace of radial functions in $\A$.

As shown in \cite{CDPSY25} (see also Lemma \ref{thmE}), for each $u\in \E$, $$\phi_u(x):=(\K_a*u^2 )(x)=\int_{\R^3}\frac{1-e^{-\frac{|x-y|}{a}}}{|x-y|}u^2(y)dy\in \A$$ is the unique weak solution 
	of $$-\Delta \phi_u+a^2\Delta^2\phi_u=4\pi u^2 \qquad \text{in } \RT.$$
So, in order to find a weak solution  to system \eqref{eq1.1}, one need to consider only the nonlocal equation
\begin{equation}\label{nonloc}
-\Delta u+q^2\phi_u u=|u|^{p-2}u \qquad\text{in }\RT.
\end{equation}
Following \cite[Section 5]{CDPSY25}, such solutions can be find as critical point of the  functional
\begin{equation*}\label{I}
	\begin{aligned}
		I_a(u) 
		&:=\frac{1}{2}\|\nabla u\|_2^2 + \frac{q^2}{4}\int_{\R^3}\int_{\R^3}\K_a(x-y)u^2(x)u^2(y)dxdy - \frac{1}{p}\|u\|^p_p,
	\end{aligned}
\end{equation*}
which is well-defined on $\E$ for each $p\in [3,6]$, and on $\E_r$ for each $p\in (18/7,6]$.
Namely if $u\in\E$ is a critical point of $I_a$, then it solves \eqref{nonloc} in a weak sense and $(u,\phi_u)$ is a weak solution of system \eqref{eq1.1}.

Based on a careful analysis on this functional framework, and by applying an approximation procedure that added a small mass term followed by passing to the limit, in \cite{CDPSY25} the authors finally proved that there exists a nontrivial weak solution in $\E$  (or $\E_r$) to system \eqref{eq1.1} for $p\in (3,6)$. However, therein, neither the existence of ground state solutions nor the asymptotic  as $a\to 0$ are considered. Moreover nonexistence results are also not mentioned.

In this paper, therefore, we aim to complete the study initiated in \cite{CDPSY25} focusing on these three aspects.

Our first result concerns the existence of ground state solution to \eqref{eq1.1}.  
Inspired by \cite{DS19,IR12} and adopting a different strategy from that in \cite{CDPSY25}, by a Mountain Pass approach  in combination with a splitting lemma, we obtain the following.

	\begin{thm}\label{Thm2}
		Let $q\ne 0$. For any $p\in (4,6)$ and $a>0$, there exists a ground state solution in $\E$ (or $\E_r$) to system \eqref{eq1.1}.
	\end{thm}

		\begin{rem}
The existence of a ground state solution in the case $p \in (3,4]$ remains an open problem due to the limitations of our approach. In particular,  the lack of homogeneity of $\K_a$ and the absence of the frequency term $\omega$ pose serious challenges to the construction of the mountain-pass structure and the arguments of \cite{IR12} do not work. 
		
	\end{rem}		
		
Our second type of results concerns the nonexistence of  nontrivial solutions for system \eqref{eq1.1}, which is established by using a Nehari-Pohozaev type argument.

\begin{thm}\label{Thm4}
Let $q\ne 0$. For any $p\ge 6$ or $p< 12/7$,  and $a>0$, system \eqref{eq1.1} admits no nontrivial solution in $\E\times\A$. 
\end{thm}

Observe that our result looks weaker compared with the corresponding one in the positive mass case studied in \cite{DS19},  where the nonexistence has been obtained for $p\ge 6$ and $p\le 2$. This is strictly related with the lack of suitable embedding properties of the space $\E$ and, as explained with more details in Remarks \ref{rem1} and \ref{rem2}, in order to recover the whole case $p\le 2$, we need to restrict ourself to the radial setting. More precisely we have the following.  

\begin{thm}\label{Thm4-bis}
Let $q\ne 0$ and $a>0$, system \eqref{eq1.1}  has no nontrivial solution in $\E_r \times \A_r$ for all $p \le 2$.
\end{thm}

In the radial setting, moreover, due to the additional informations about the embedding properties of $\E_r$, something more can be obtained. In particular, if $u$ is a ground state solution with energy $c_a$, then as shown in Corollary~\ref{lem9}, it admits the following minimax characterization:
			$$0<c_a=I_a(u)=\inf_{v\in\M_{a,r}}I_a(v)=\inf_{v\in \E_r\setminus\{0\} }\max_{t>0}I_a(t^2v(tx)).$$
		Here, $\M_{a,r}$ denotes a Nehari-Pohozaev type set, see \eqref{M} for a precise definition.

Based on this discovery, now
we can analyse the limiting behaviour of ground state solutions to system \eqref{eq1.1} as $a\to0$, thereby establishing a connection with system \eqref{SP}.

\begin{thm}\label{Thm1}
	Assume $q\ne0$, $p\in (4,6)$ and $a>0$. Let $(u_a,\phi_a)\in \E_r\times\mathcal{A}_r$ be a ground state solution to system \eqref{eq1.1}. Then, as $a\to 0$, one has
	$$u_a\weakto u_0 \text{ weakly in } D_r^{1,2}(\RT) ~~\text{ and }~~ \phi_a\weakto \phi_0 \text{ weakly in } D_r^{1,2}(\RT),$$
	where $(u_0,\phi_0)\in D_r^{1,2}(\RT)\times D_r^{1,2}(\RT)$
	solves system \eqref{SP} in the following weak sense:
	 $$\int_{\RT}\nabla\phi_0 \cdot \nabla\eta dx=4\pi\int_{\RT} u_0^2\eta dx,~~ \forall \eta\in C_{c}^\infty(\RT),$$
	 and
	 $$\int_{\RT}\nabla u_0 \cdot\nabla \varphi dx+q^2 \int_{\RT}\phi_{0}u_0\varphi dx=\int_{\RT}|u_0|^{p-2}u_0\varphi dx,~~ \forall \varphi\in C_{c}^\infty(\RT).$$
	
\end{thm}

\begin{rem}
Actually if  $(u_a,\phi_a)\in \E_r\times\mathcal{A}_r$ is a nontrivial weak solution of \eqref{eq1.1}, with $a>0$, and such that the family $\{I_a(u_a)\}_{a>0}$ is bounded, 
we get the same result with similar arguments.
\end{rem}

This paper is organized as follows. 
In Section \ref{se2}, we recall the functional setting introduced in \cite{CDPSY25} and we provide and establish some preliminary results that will be useful in the sequel.
Section \ref{se3} is devoted to proving the existence of ground state solutions via the Mountain Pass Theorem, where the compactness can be recovered with the help of a splitting lemma.
In Section \ref{se4}, we focus on the regularity of nontrivial weak solutions and establish some nonexistence results for system \eqref{eq1.1} by means of Pohozaev identity. 
Finally, in Section \ref{se5}, we study the asymptotic behaviour of the radial ground state solutions we obtain as $a\to 0$.
	
We conclude with a list of notations:
	\begin{itemize}
		\item $\|\cdot\|_q$ denotes the usual $L^q(\R^3)$ norm for $1\le q\le +\infty$;
        \item $o_n(1)$ indicates a vanishing sequence as $n\to+\infty$;
		\item the subscript $r$ denotes the respective subspace of radial functions;
		\item $C$ is used to represent suitable positive constants whose value may be distinct from line to line.
	\end{itemize}
	
	In what follows, unless otherwise indicated, we shall always assume that $q \neq 0$ and $a>0$.

\section{Preliminaries}\label{se2}
This section aims at introducing the functional setting and presenting some preliminary results. We start by 
stating the characterization of the functional spaces $\E$ and $\E_r$ along with some of their useful properties.
\begin{lem}(\cite[Theorem 1.1]{CDPSY25})\label{thmE}
	We have
	\[
	\E=\{u\in D^{1,2}(\RT):\phi_u\in \A\},
	\]
	where
	\begin{equation*}\label{eq2.3}
		\phi_u(x):=(\K_a*u^2 )(x)=\int_{\R^3}\frac{1-e^{-\frac{|x-y|}{a}}}{|x-y|}u^2(y)dy.
	\end{equation*}
	Moreover, for each $u\in \E$,  $\phi_u$ is the unique weak solution in $\mathcal A$
	of $-\Delta \phi_u+a^2\Delta^2\phi_u=4\pi u^2$ in $\RT$,
	and
	\[
	\|\phi_u\|^2_\mathcal A = 	\int_{\mathbb{R}^3} (|\nabla \phi_u|^2 + a^2 |\Delta \phi_u|^2) dx = 4\pi\int_{\R^3} \phi_u u^2 dx
	=4\pi V(u^2,u^2).
	\]
\end{lem}

\begin{lem}(\cite[Proposition 4.1]{CDPSY25})\label{PropA.2}
	The space $(\E,\|\cdot\|_{\E})$ is a uniformly convex Banach space.
\end{lem}

\begin{lem}(\cite[Proposition 4.2]{CDPSY25})\label{density}
	The space  $C_c^{\infty}(\R^3)$ is dense in $(\E,\|\cdot\|_\E)$.
\end{lem}

\begin{lem}(\cite[Lemma 4.3]{CDPSY25})\label{lem3.4}
	Let $\{u_n\}$ be a sequence in $\E$. We have that
	\begin{enumerate}[label=(\roman*),ref=\roman*]
		\item \label{weakconv} if $u_n\rightharpoonup u \text{ weakly in } \E$, then
		$\phi_{u_n}\rightharpoonup \phi_{u}$  weakly in $\mathcal{A}$;
		\item \label{stronvconv}  $u_n\to u$ in $\E$ if and only if
		$u_n\to u$ in $D^{1,2}(\RT)$
		and $\phi_{u_n}\to \phi_{u}$ in $\A$.
	\end{enumerate}
\end{lem}

\begin{lem}(\cite[Theorem 1.2]{CDPSY25})\label{tE36}
	The space $\E$ is continuously embedded into $L^\tau(\RT)$, for any $\tau\in [3,6]$.
\end{lem}

\begin{lem}(\cite[Theorem 1.3]{CDPSY25})\label{tEr187}
	The space $\E_r$ is continuously embedded into $L^\tau(\RT)$, for any $\tau\in (18/7,6]$. The embedding is compact for any $\tau\in (18/7,6)$.
\end{lem}

Moreover, the Bopp-Podolsky energy $V(f,g)$ satisfies the following Cauchy-Schwarz type inequality.

\begin{lem}(\cite[Proposition 3.4]{CDPSY25})\label{prop:main}
	For all  functions $f,g:\RT\to\R$, with $V(|f|, |f|), V(|g|, |g|)<+\infty$, we have that
	\begin{equation*}\label{CS}
		|V(f,g)|^2\leq V(f,f)V(g,g).
	\end{equation*}
\end{lem}	

In addition, the following inequality holds:
\begin{lem}(\cite[Lemma 5.4]{CDPSY25})\label{lem2.3}
	Let $b\ge0$. Then one has
	\begin{equation}\label{eq3.24}
		t^3(e^{-\frac{b}{t}}-e^{-b})+\frac{1-t^3}{3}be^{-b}\ge0, ~~\forall t>0.
	\end{equation}
	
\end{lem}

 Meanwhile, the following statements are valid:
\begin{lem}(\cite[Proposition 5.2]{CDPSY25})\label{IC1}
	For each $p\in[3,6]$, we have that $I_a\in C^2(\E)$, and for any $u, v\in \E$
	\begin{equation*}\label{dI}
		I_a'(u)[v] = \int_{\R^3} \nabla u \cdot \nabla v dx + q^2 \int_{\R^3} \phi_u u v dx - \int_{\R^3} |u|^{p-2}u v dx.
	\end{equation*}
	The same result holds on $\mathcal E_r$, for $p\in(18/7,6]$.
\end{lem}

\begin{cor}(\cite[Corollary 5.3]{CDPSY25})\label{ws}
	Let $p\in[3,6]$ (resp. $p\in(18/7,6]$) and $u\in \E$ (resp. $u\in \E_r$). Then $(u,\phi_u)$   is a weak solution of the system  \eqref{eq1.1} if and only if $I_a'(u)=0$. 	
\end{cor}

This section concludes with some auxiliary lemmas. We define $T:\E^4\to \R$ by
$$T(u,v,\omega,z):=\int_{\R^3}\int_{\R^3}\K_a(x-y)u(x)v(x)\omega(y)z(y)dxdy.
$$ 
Clearly, $T$ is linear and continuous with respect to each variable, where the continuity follows from H\"older inequality and Lemma \ref{prop:main}. 

Analogously to \cite[Lemma 2.3]{IR12}, we have that
\begin{lem}\label{lem1}
	Assume that there exist three weakly convergent sequences in $\E$, namely, $u_n\weakto u$, $v_n\weakto v$, $\omega_n\weakto \omega$, and $z\in\E$. Then 
	$$T(u_n,v_n,\omega_n,z)\to T(u,v,\omega,z).$$
\end{lem}
\begin{proof}
	Based on Lemma \ref{thmE}, Lemma \ref{lem3.4}, and Lemma \ref{prop:main}, this lemma can be proved by an argument similar to that of \cite[Lemma 2.3]{IR12}, with only minor modifications. Hence, for the sake of brevity, we omit the proof.
	
\end{proof}

Relying on \cite[Lemma I.1]{L84} and Lemma \ref{tE36}, we further obtain the following Vanishing lemma.
\begin{lem}\label{lem3}
	Let $\{u_n\}$ be a bounded sequence in $\E$, $q\in[3,6)$, and assume that
	$$\sup_{y\in \R^3}\int_{B_R(y)}|u_n|^qdx\to0, \quad \text{ for some } R>0.$$
	Then, $u_n\to 0$ in $L^\alpha(\R^3)$ for any $\alpha\in(3,6)$.
\end{lem}

\begin{proof}
	This lemma follows by reasoning analogous to the proof of \cite[Lemma 2.4]{IR12}, up to minor changes. For this reason, we skip the proof.

\end{proof}

\section{Existence of ground state solutions}\label{se3}
	In this section, we focus on proving Theorem \ref{Thm2}, that is, the existence of ground state solutions of system \eqref{eq1.1} for the case $p\in (4,6)$. Let us define $M:\E\to \R$ as:
$$M[u]:=\|\nabla u\|_2^2+\int_{\R^3}\int_{\R^3}\K_a(x-y)u^2(x)u^2(y)dxdy.
$$
Just by taking into account the definitions of $M$ and $\|\cdot\|_{\E}$, we can easily check that for any $u\in \E$
\begin{equation}\label{neweq1}
	\frac{1}{2}\|u\|_{\E}^4\le M[u]\le \|u\|_{\E}^2, \quad \text{ if either } \|u\|_{\E}\le 1 \text{ or } M[u]\le1.
\end{equation}

We now turn our attention to the functional $I_a$, and show that it satisfies the geometric properties of the Mountain Pass Theorem.
\begin{lem}\label{MP}
Let $p\in (4,6)$. Then we have that 
	\begin{enumerate}[label=(\roman*),ref=\roman*]
\item \label{i} $I_a(0)=0$;
\item \label{ii} there exist $\delta,\rho>0$ such that $I_a(u)>\delta$ for all $u\in \E$ with $\|u\|_{\E}=\rho$;
\item \label{iii} there exists $\omega\in\E$ with  $\|\omega\|_{\E}>\rho$ such that $I_a(\omega)<0$.
	\end{enumerate}
\end{lem}

\begin{proof}
Property \eqref{i} is trivial. 

Moreover, by Lemma \ref{tE36}, it follows that 
\begin{equation*}\label{eq7.1}
\begin{aligned}
I_a(u) 
&\ge\min\left\{\frac{1}{2},\frac{q^2}{4}\right\}M[u]-\frac{C}{p}\|u\|_{\E}^p,
\end{aligned}
\end{equation*}	
which, together with \eqref{neweq1}, indicates that \eqref{ii} holds if we take $\|u\|_{\E}=\rho>0$ small enough.

To prove \eqref{iii}, we let $u_t:=t u$ with $u\in \E$ and $t\ge 0$. This yields that
\begin{equation}\label{eq7.2}
	\begin{aligned}
		I_a(u_t)
		& =\frac{t^2}{2}\|\nabla u\|_2^2+\frac{q^2t^4}{4}\int_{\R^3}\int_{\R^3}\K_a(x-y)u^2(x)u^2(y)dxdy-\frac{t^{p}}{p}\|u\|_p^p.
	\end{aligned}
\end{equation}	
Therefore, the desired results follow considering $t\to +\infty$. 
	
\end{proof}

In virtue of Lemma \ref{MP} and Mountain Pass Theorem, we define the Mountain Pass level for $I_a$ as follows: 
\begin{equation}\label{eq7.15}
	c_a:=\inf_{\gamma\in \Gamma}\max_{t\in[0,1]}I_a(\gamma(t))>0,
\end{equation}
where $\Gamma:=\{\gamma \in C([0,1],\E):\gamma(0)=0, I_a(\gamma(1))<0\}$. Then there exists a Palais-Smale (PS) sequence $\{u_n\}\subset\E$ for $I_a$ at level $c_a$ such that, as $n\to +\infty$,
\begin{equation}\label{eq7.9}
	I_a(u_n)\to c_a ~\text{ and }~ I_a'(u_n)\to 0.
\end{equation}
 
Our intention is to find a critical point at level $c_a$. By the next proposition, this weak solution of system \eqref{eq1.1} will be a ground state, namely, a solution with minimal energy among all nontrivial weak solutions.

\begin{prop}\label{prop1}
	Suppose $p\in (4,6)$ and $u\in \E\backslash\{0\}$ is a weak solution of system \eqref{eq1.1}. Then it holds  
	$$I_a(u)\ge c_a.$$
\end{prop}
 
\begin{proof}
	Given such solution $u\in \E\setminus \{0\}$, let us define again $u_t:=tu$ with $t\ge 0$, and $\bar{\gamma}:\R^+\to \E$, $\bar{\gamma}(t)=u_t$. Clearly, $\bar{\gamma}(0)=0$ and $\bar{\gamma}$ is a continuous curve in $\E$. 
	Moreover, by \eqref{eq7.2}, we have that
	\begin{equation*}
		\begin{aligned}
			f(t):=I_a(u_t) =\frac{t^2}{2}\|\nabla u\|_2^2+\frac{q^2t^4}{4}\int_{\R^3}\int_{\R^3}\K_a(x-y)u^2(x)u^2(y)dxdy-\frac{t^{p}}{p}\|u\|_p^p.
		\end{aligned}
	\end{equation*}
	It is easy to check that $f$ is of class $C^1$ and has a unique critical point that corresponds to its maximum by noting the following facts:
	\begin{equation}\label{eq7.29}
		\begin{aligned}
			f'(t)=t\left(\|\nabla u\|_2^2+q^2t^2\int_{\R^3}\int_{\R^3}\K_a(x-y)u^2(x)u^2(y)dxdy-t^{p-2}\|u\|_p^p\right)=:tg(t),
	\end{aligned}
\end{equation}
and 
\begin{equation*}
	\begin{aligned}
		g'(t)=t\left(2q^2\int_{\R^3}\int_{\R^3}\K_a(x-y)u^2(x)u^2(y)dxdy-(p-2)t^{p-4}\|u\|_p^p\right).
	\end{aligned}
\end{equation*}
	Recalling now that $u$ is a nontrivial weak solution (i.e., $I_a'(u)[u]=0$),
	we deduce from \eqref{eq7.29} and Lemma \ref{IC1}  that 
	$$f'(1)=\|\nabla u\|_2^2+q^2\int_{\R^3}\int_{\R^3}\K_a(x-y)u^2(x)u^2(y)dxdy-\|u\|_p^p=I_a'(u)[u]=0.$$
	Therefore,  
	\begin{equation}\label{eq7.20}
		\max_{t\in \R^+}I_a(\bar{\gamma}(t))=\max_{t\in \R^+}I_a(u_t)=I_a(u).
	\end{equation}

	On the other hand, being $\lim\limits_{t\to +\infty}f(t)=-\infty$,  we can take a constant $T_0>0$ such that 
	$$I_a(\bar{\gamma}(T_0))<0.$$ 
	Reparametrizing $\gamma:[0,1]\to \E$ by $\gamma(t):=\bar{\gamma}(T_0t)$, we obtain that $\gamma\in \Gamma$. Thus, we derive from \eqref{eq7.15} and \eqref{eq7.20} that  
	$$c_a\le \max_{t\in [0,1]}I_a(\gamma(t))=\max_{t\in [0,1]}I_a(\bar{\gamma}(T_0t))\le \max_{t\in \R^+}I_a(\bar{\gamma}(t))=I_a(u).$$
	This completes the proof.
	
\end{proof}
 
In addition, we derive the following estimates concerning the nontrivial weak solutions of system \eqref{eq1.1}.
 \begin{lem}\label{lem5}
 	Let $p\in (4,6)$. There exists a constant $C>0$ such that $$M[u]\ge C,$$ for any $u\in \E$ nontrivial weak solution of system \eqref{eq1.1}.
 \end{lem}
 
 \begin{proof}
 Let $u\in \E$ be a nontrivial weak solution of system \eqref{eq1.1}. Without loss of generality we can suppose that $M[u]\le 1$. Hence by \eqref{neweq1} and Lemma \ref{tE36}, we have
 	\begin{equation*}
 		\begin{aligned}
 			\min\{1,q^2\}M[u]
 			\le \|\nabla u\|_2^2+q^2\int_{\R^3}\int_{\R^3}\K_a(x-y)u^2(x)u^2(y)dxdy
 			=\|u\|_p^p\le C\|u\|_{\E}^p,
 		\end{aligned}
 	\end{equation*}
which concludes the proof. 	

 \end{proof}

The next lemma describes the behaviour of (PS) sequences.
\begin{lem}\label{lem2}
	For any $p\in(4,6)$, let $\{u_n\}\subset\E$ be a (PS) sequence for $I_a$  at level $c$. Then the sequence $\{u_n\}$ is bounded in $\E$, and its weak limit in $\E$ is a critical point of $I_a$.
\end{lem}

\begin{proof}
	Observe that
	\begin{equation*}
		\begin{aligned}
			p(c+1)+\|u_n\|_{\E}
			&\ge pI_a(u_n)-I_a'(u_n)[u_n]\\
			&= \left(\frac{p}{2}-1\right)\|\nabla u_n\|_2^2 + \left(\frac{p}{4}-1\right)q^2\int_{\R^3}\int_{\R^3}\K_a(x-y)u_n^2(x)u_n^2(y)dxdy, 
		\end{aligned}
	\end{equation*}
	which, together with the definition of the norm in $\E$, implies that $\{u_n\}$ is bounded in $\E$.
	
	Recalling that $\E$ is a reflexive Banach space in view of  Lemma \ref{PropA.2}, there exists $v_0\in \E$ such that, up to a subsequence, as $n\to +\infty$, we have that $u_n\weakto v_0$ weakly in $\E$, and $u_n\to v_0$ strongly in $L_{\rm loc}^\tau(\R^3)$ for all $\tau\in[1,6)$. Then, for all $\psi\in C_c^{\infty}(\R^3)$, we get 
	$$\int_{\R^3}\nabla u_n \nabla \psi dx\to \int_{\R^3}\nabla v_0 \nabla \psi dx,
	\quad 
	\int_{\R^3}|u_n|^{p-2}u_n\psi dx \to \int_{\R^3}|v_0|^{p-2}v_0\psi dx.$$
	Moreover, by using Lemma \ref{lem1}, there holds 
	$$\int_{\R^3}\phi_{u_n}u_n\psi dx \to\int_{\R^3}\phi_{v_0}v_0\psi dx.$$
	Consequently, as $n\to +\infty$, one can see that
	\begin{equation*}
		\begin{aligned}
			0
			&\gets I_a'(u_n)[\psi]+o_n(1)\\
			&= \int_{\R^3} \nabla u_n  \nabla \psi dx + q^2 \int_{\R^3} \phi_{u_n} u_n \psi dx - \int_{\R^3} |u_n|^{p-2}u_n \psi dx+o_n(1)\\
			&= \int_{\R^3}\nabla v_0 \nabla \psi dx+q^2\int_{\R^3}\phi_{v_0}v_0\psi dx-\int_{\R^3}|v_0|^{p-2}v_0\psi dx= I_a'(v_0)[\psi],
		\end{aligned}
	\end{equation*}
	that is, 
	$$\int_{\R^3}\nabla v_0 \nabla \psi dx+q^2\int_{\R^3}\phi_{v_0}v_0\psi dx=\int_{\R^3}|v_0|^{p-2}v_0\psi dx.$$ 
	By Lemma \ref{density}, the previous identity holds also for every $\psi\in \E$ and so $v_0$ is a critical point for $I_a$.

\end{proof}

As a consequence of the second part of Lemma \ref{lem2}, if the weak limit is indeed zero, the preceding proof further yields the corollary below.

\begin{cor}\label{cor1}
	Let $\{u_n\}\subset\E$ be a sequence such that $u_n\weakto0$ weakly in $\E$. Then we have, as $n\to +\infty$, 
	$$I_a'(u_n)\to 0.$$
\end{cor}

Building on Lemma \ref{lem5} and Lemma \ref{lem2}, we are now ready to  prove the Splitting lemma, which analyzes the behavior of (PS) sequences and plays a crucial role in recovering compactness. For this purpose, we define the functional  ${M}_q:\E\to \R$ by
$${M}_q[u]:=\|\nabla u\|_2^2+q^2\int_{\R^3}\int_{\R^3}\K_a(x-y)u^2(x)u^2(y)dxdy.$$
Then we obtain the following result. 
\begin{lem}[Splitting lemma]\label{Thm3}
	Let $p\in(4,6)$ and $\{u_n\}\subset \E$ be a (PS) sequence of $I_a$ at a certain level $c>0$. 
	Then, up to subsequences, there exist $k\in\N\cup\{0\}$ and a finite sequence
	$$(v_0,v_1,\dots,v_k)\subset\E, \quad v_i\not\equiv0  \text{ for } i\ge 0$$
of weak solutions 
$$-\Delta u +q^2\phi_u u=|u|^{p-2}u \qquad\text{in }\RT,$$ 
and $k$ sequences $\{\xi_n^1\},\dots,\{\xi_n^k\}\subset\R^3$, such that, as $n\to +\infty$, 
$$\left\|u_n-v_0-\sum_{i=1}^{k}v_i(\cdot -\xi_n^i)\right\|_{\E}\to 0;$$
$$|\xi_n^i|\to +\infty, \quad |\xi_n^i-\xi_n^j|\to +\infty, \quad i\neq j ;$$
$$\sum_{i=0}^{k}I_a(v_i)=c, \quad {M}_q[u_n]\to\sum_{i=0}^{k}{M}_q[v_i].$$ 
\end{lem}

\begin{proof} The proof is carried out in the following steps.
	
\textbf{Step 1.}  Passing to a subsequence if necessary and applying Lemma \ref{lem2}, we deduce that there exists $v_0\in\E$ such that
$$u_n \weakto v_0 \text{ weakly in } \E \text{ as } n\to +\infty,$$
and
\begin{equation}\label{eq7.6}
	I_a'(v_0)=0.
\end{equation}

Set 
$u_{n,1}:=u_n-v_0.$
We claim now that
\begin{equation}\label{eq7.4}
	I_a(u_n)-I_a(u_{n,1})\to I_a(v_0),\quad {M}_q[u_n]-{M}_q[u_{n,1}]\to {M}_q[v_0].
\end{equation} 
By using the weak convergence in $D^{1,2}(\R^3)$, the well-known Brezis-Lieb lemma and Lemma \ref{lem1}, we obtain that
$$\|\nabla u_n\|_2^2-\|\nabla (u_n-v_0)\|_2^2 \to \|\nabla v_0\|_2^2,\quad
\|u_n\|_p^{p}-\|u_n-v_0\|_p^{p} \to \|v_0\|_p^{p},$$
and 
\begin{equation*}
	\begin{aligned}
		&T(u_{n,1},u_{n,1},u_{n,1},u_{n,1})\\
		&=T(u_n,u_n-v_0,u_n-v_0,u_n-v_0)-T(v_0,u_n-v_0,u_n-v_0,u_n-v_0)\\
		&=T(u_n,u_n-v_0,u_n-v_0,u_n-v_0)+o_n(1)\\
		&=T(u_n,u_n,u_n-v_0,u_n-v_0)+o_n(1)\\
		&=T(u_n,u_n,u_n,u_n-v_0)+o_n(1)\\
		&=T(u_n,u_n,u_n,u_n)-T(v_0,v_0,v_0,v_0)+o_n(1),
	\end{aligned}
\end{equation*}
which, combined with the definitions of $I_a$ and ${M}_q$, indicates that the claim \eqref{eq7.4} holds.

If $u_{n,1}\to 0$ in $\E$, then the proof is concluded. To be specific, in this case, we have $I_a(u_{n,1})\to 0$ and ${M}_q[u_{n,1}]\to 0$. Thus, by \eqref{eq7.4}, it holds
$$0<c=I_a(u_n)+o_n(1)\to I_a(v_0), \quad {M}_q[u_n]\to {M}_q[v_0].$$
Furthermore,  applying \eqref{eq7.6} and the fact $I_a(v_0)>0$, it follows that $v_0\not\equiv0$ is a nontrivial weak solution of \eqref{nonloc}.

Suppose now that $u_{n,1}\not\to 0$ in $\E$.  
Recalling that $\{u_n\}\subset \E$ is a (PS) sequence for $I_a$, we derive from \eqref{eq7.6} that
\begin{equation}\label{eq7.3}
		{M}_q[u_n] - \|u_n\|_p^{p}=I_a'(u_n)[u_n]
		\to0=I_a'(v_0)[v_0]= {M}_q[v_0]-\|v_0\|_p^{p}.
\end{equation}
Noting that $u_n\weakto v_0$ weakly in $\E$, together with Lemma \ref{lem3.4} \eqref{weakconv}, we get 
\begin{equation}\label{eq7.28}
	u_n\weakto v_0 \text{ weakly in } D^{1,2}(\R^3) ~\text{ and }~ \phi_{u_n}\rightharpoonup \phi_{v_0} \text{ weakly in } \mathcal{A}.
\end{equation}
Since $u_n\not\to v_0$ in $\E$,  \eqref{eq7.4} shows that at least one of the convergences in \eqref{eq7.28} is not strong convergence. Combined with Lemma \ref{thmE}, up to a subsequence, this yields that
\begin{equation*}
	\begin{aligned}
		{M}_q[u_n]+o_n(1)&=\|\nabla u_n\|_2^2+q^2\int_{\R^3}\int_{\R^3}\K_a(x-y)u_n^2(x)u_n^2(y)dxdy+o_n(1)\\
		&=\|\nabla u_n\|_2^2+\frac{q^2}{4\pi}\|\phi_{u_n}\|_{\A}^2+o_n(1)\\
		&> \|\nabla v_0\|_2^2+\frac{q^2}{4\pi}\|\phi_{v_0}\|_{\A}^2={M}_q[v_0].
	\end{aligned}
\end{equation*}
Combining this with \eqref{eq7.3}, we conclude that $u_n\not\to v_0$ in $L^p(\R^3)$, that is,
$u_{n,1}\not\to 0 \text{ in } L^p(\R^3).$
By Lemma \ref{lem3}, for any $q\in[4,6)$, we conclude that there exist $\delta_1>0$ and $\{\xi_n^1\}\subset\R^3$, such that
\begin{equation}\label{eq7.5}
	\lim_{n \to +\infty}\int_{B_1(\xi_n^1)}|u_{n,1}(x)|^qdx=\lim_{n \to +\infty}\int_{B_1(0)}|u_{n,1}(x+\xi_n^1)|^qdx\ge\delta_1>0.
\end{equation}	
Since $u_{n,1}\weakto 0$ weakly in $\E$, then $u_{n,1}\to 0$ in $L_{\rm loc}^\tau(\R^3)$ for all $\tau\in[1,6)$, we thus deduce by contradiction that $|\xi_n^1|\to +\infty$.

\textbf{Step 2.} We now turn our attention to the sequence $\{u_{n,1}(\cdot +\xi_n^1)\}$. Clearly, it is a (PS) sequence at level $c-I_a(v_0)$ by applying \eqref{eq7.4} and Corollary \ref{cor1}. Going if necessary to a subsequence, we may assume that 
\begin{equation}\label{eq7.7}
	u_{n,1}(\cdot +\xi_n^1)\weakto v_1 \text{ weakly in } \E.
\end{equation}
It then follows from Lemma \ref{lem2} and \eqref{eq7.5} that $v_1\not \equiv0$ is a nontrivial weak solution of \eqref{nonloc}.

Define $u_{n,2}:=u_{n,1}-v_1(\cdot -\xi_n^1).$
Arguing as in Step 1 and in view of \eqref{eq7.4}, one has 
\begin{equation}\label{eq7.21}
	I_a(u_{n,2})=I_a(u_{n,1})-I_a(v_1)+o_n(1)=I_a(u_n)-I_a(v_0)-I_a(v_1)+o_n(1),
\end{equation}
and
\begin{equation*}
{M}_q[u_{n,2}]={M}_q[u_{n,1}]-{M}_q[v_1]+o_n(1)=	 {M}_q[u_n]-{M}_q[v_0]-{M}_q[v_1]+o_n(1).
\end{equation*}
Furthermore, one can see that $u_{n,2}=u_{n,1}-v_1(\cdot -\xi_n^1)\weakto0 $ weakly in $\E$ since both summands converge weakly to zero, and $u_{n,2}(\cdot +\xi_n^1)=u_{n,1}(\cdot +\xi_n^1)-v_1\weakto 0$ weakly in $\E$ by using \eqref{eq7.7}.

If $u_{n,2}\to 0$ in $\E$, then we are done. Otherwise, as in Step 1 (see \eqref{eq7.3}-\eqref{eq7.5}), we can show that $u_{n,2}\not\to 0$ in $L^p(\R^3)$. Consequently, for each $q\in[4,6)$, there exist $\delta_2>0$ and $\{\xi_n^2\}\subset\R^3$, such that
\begin{equation}\label{eq7.8}
	\lim_{n \to +\infty}\int_{B_1(\xi_n^2)}|u_{n,2}(x)|^qdx=\lim_{n \to +\infty}\int_{B_1(0)}|u_{n,2}(x+\xi_n^2)|^qdx\ge\delta_2>0.
\end{equation}	
Recalling that $u_{n,2}\weakto0 $ weakly in $\E$ and $u_{n,2}(\cdot +\xi_n^1)\weakto 0$ weakly in $\E$, it then follows from \eqref{eq7.8} that 
$$|\xi_n^2|\to +\infty ~\text{ and }~ |\xi_n^2-\xi_n^1|\to +\infty.$$
Furthermore, up to subsequences, we derive from \eqref{eq7.21}, \eqref{eq7.8} and Lemma \ref{lem2} that 
$$u_{n,2}(\cdot +\xi_n^2)\weakto v_2\not\equiv0 \text{ weakly in } \E,$$ 
where $v_2\not \equiv0$ is a nontrivial weak solution of $-\Delta u +q^2\phi_u u=|u|^{p-2}u$.
Let
$$u_{n,3}:=u_{n,2}-v_2(\cdot -\xi_n^2).$$

\textbf{Step 3.} Iterating the above procedure, we construct sequences $\{u_{n,j}\}_{j=1,2,\dots}$ and $\{\xi_n^j\}_{j=1,2,\dots}$ such that
$$u_{n,j+1}:=u_{n,j}-v_j(\cdot -\xi_n^j),$$
$$u_{n,j}(\cdot +\xi_n^j)\weakto v_j \text{ weakly in } \E,$$
$$I_a'(v_j)=0, \quad v_j\not\equiv0 \text{ for } j\ge 1,$$
$$I_a(u_{n,j})=I_a(u_n)-\sum_{i=0}^{j-1}I_a(v_i)+o_n(1),$$
and
$${M}_q[u_{n,j}]={M}_q[u_n]-\sum_{i=0}^{j-1}{M}_q[v_i]+o_n(1).$$
Noting that ${M}_q[u_n]$ is bounded and that ${M}_q[v_i]$ is away from zero by Lemma \ref{lem5}, this means that the iteration must stop at a certain point, that is, for some $k$, $u_{n,k}\to 0$ in $\E$. This concludes the proof.

\end{proof}

Finally, with the help of the above lemmas, we proceed to the proof of   Theorem \ref{Thm2}.
\begin{proof}[Proof of Theorem \ref{Thm2}] 
In view of Lemma~\ref{lem2}, we take $\{u_n\}\subset\E$ to be the bounded (PS) sequence for $I_a$ at level $c_a>0$ as in \eqref{eq7.9}. Applying Lemma \ref{Thm3}, we obtain that
	$$\sum_{i=0}^{k}I_a(v_i)=c_a,$$
where each $v_i $ is  weak solution of \eqref{nonloc}  and only $v_0$ could be zero. 
By Proposition \ref{prop1}, this yields that $I_a(v_i)\ge c_a$ whenever $v_i\ne0$. Thus, there are two possibilities: 
$$\text{either }~ v_0\ne 0 \text{ and } k=0, 
~\text{ or }~ 
v_0=0 \text{ and } k=1.$$
In the first case, $v_0$ is a ground state solution at level $c_a$ and $u_n\to v_0$ in $\E$. In the latter, $v_1$ is a ground state solution at level $c_a$ and $u_{n}(\cdot+\xi_n^1)=u_{n,1}(\cdot+\xi_n^1)\to v_1$ in $\E$. 

The proof is thus complete.

\end{proof}

By using a similar reasoning as the above proof of  Theorem~\ref{Thm2},  together with  Lemma \ref{tEr187}, Lemma \ref{MP} and Proposition \ref{prop1}, we can  more directly conclude that system \eqref{eq1.1} has a ground state solution in $\E_r$ for $p\in (4,6)$.

\section{Pohozaev identity and nonexistence results}\label{se4}
This section is devoted to establishing certain nonexistence results. 
To this end, we begin studying the regularity of solutions of system \eqref{eq1.1}.

\begin{lem}\label{reg}
Suppose $p\in[3,6)$ and $(u,\phi_u)\in \E\times \A$ is a nontrivial weak solution of system \eqref{eq1.1}, then $(u,\phi_u)\in C_{\rm loc}^{2,\sigma}(\RT)\times C_{\rm loc}^{4,\sigma}(\RT)$, for $\sigma\in (0,1/2]$.    
\end{lem}

\begin{proof}
    Recalling the proof of \cite[Lemma 1.30]{W96}, we know that 
if $-\Delta u=bu$ with $b\in L^{3/2}_{\rm loc}(\R^3)$, then $u\in C^2(\R^3)$. \\
	In our case, let $b:=|u|^{p-2}-q^2\phi_u$. Then, by \eqref{eq7.22} and Lemma \ref{tE36}, one can see that 
	$$\phi_u\in \A\hookrightarrow L^{\infty}(\R^3),$$ 
	and 
	$$|u|^{p-2}\in L^{6/(p-2)}(\R^3), \text{ where } 3/2<6/(p-2).$$ 
	Clearly, $b\in L^{3/2}_{\rm loc}(\R^3)$. By applying an analogous argument as in  \cite[\textit{A.1. Regularity of the solutions}]{DS19}, we deduce that the weak solutions obtained are indeed classical.
    
\end{proof}

Moreover the following Nehari and Pohozaev type identities hold.

\begin{lem}\label{lem6}
Suppose $p\in[3,6)$ and $(u,\phi_u)\in \E\times \A$ is a nontrivial weak solution of system \eqref{eq1.1}. Then we have that 
$$\|\nabla u\|_2^2+q^2\int_{\RT}\phi_u u^2dx=\|u\|_p^p
~~\text{ and }~~
\|\nabla \phi_u\|_2^2+a^2\|\Delta \phi_u\|_2^2=4\pi\int_{\RT}\phi_u u^2dx,$$
which are usually called Nehari identities. Furthermore, $(u,\phi_u)$ satisfies also the following Pohozaev identity
\begin{equation*}
	-\frac{1}{2}\|\nabla u\|_2^2+\frac{q^2}{16\pi}\|\nabla \phi_u\|_2^2-\frac{q^2a^2}{16\pi}\|\Delta \phi_u\|_2^2-\frac{3q^2}{2}\int_{\RT}\phi_u u^2dx+\frac{3}{p}\|u\|_p^p=0.
\end{equation*}
\end{lem}

\begin{proof}
	Multiplying the first equation of \eqref{eq1.1} by $u$ and the second equation by  $\phi_u$, and integrating on $\R^3$, we obtain Nehari identities. 	
	
	 Subsequently, by Lemma \ref{reg} and following the reasoning in \cite[\textit{A.3. The Pohozaev identity}]{DS19}, we establish the Pohozaev identity.
     
\end{proof}

\begin{rem}\label{rem1}
	In particular, for the radial space $\E_r$, the previously mentioned Pohozaev identity can be written as  
	\begin{equation}\label{eq7.10}
		P_a(u):=-\frac{1}{2}\|\nabla u\|_2^2-\frac{q^2}{4a}\int_{\RT}\int_{\RT}\left[\frac{5(1-e^{-\frac{|x-y|}{a}})}{|x-y|/a}+e^{-\frac{|x-y|}{a}}\right]u^2(x)u^2(y)dxdy
		+\frac{3}{p}\|u\|_p^p=0,
	\end{equation}
which will play a crucial role in the asymptotic behaviour discussed in Section \ref{se5}.
	Indeed, to be specific, for any $u\in\E_r$, using Lemma \ref{tEr187}
	($\E_r\hookrightarrow L^{\tau}(\R^3)$ with $\tau\in(18/7,6]$) and recalling the proof of equality (A.3) in \cite{DS19}, we obtain that
\begin{equation}\label{deltapo}
	\|\Delta\phi_u\|_2^2=\frac{2\pi}{a^3}\int_{\RT}\int_{\RT}e^{-\frac{|x-y|}{a}}u^2(x)u^2(y)dxdy,
\end{equation}
	which, together with the above Nehari identities, yields the desired result.

\end{rem}

\begin{rem}\label{rem2}
Observe that in the positive mass case, in \cite{DS19}, the authors do not need to consider just the radial setting to obtain the corresponding version of Pohozaev identity \eqref{eq7.10} which is, actually,  satisfied also in $H^1(\RT)$. In our case, instead, since the embeddings of $\E$ into $L^{\tau}(\R^3)$ with $\tau\in[2,3)$ are not known,  in order to prove \eqref{deltapo}, the radiality seems to be somehow necessary.
\end{rem}

We now turn to establishing nonexistence results  for system \eqref{eq1.1}. 
\begin{proof}[The proof of Theorem \ref{Thm4}]

	If $(u,\phi_u)\in \E\times \A$ is a nontrivial solution of system \eqref{eq1.1},  then  we deduce from Lemma \ref{lem6} that
	\begin{equation*}
		\begin{aligned}
				0&=-\frac{1}{2}\|\nabla u\|_2^2+\frac{q^2}{16\pi}\|\nabla \phi_u\|_2^2-\frac{q^2a^2}{16\pi}\|\Delta \phi_u\|_2^2-\frac{3q^2}{2}\int_{\RT}\phi_u u^2dx+\frac{3}{p}\|u\|_p^p\\
				&=-\frac{1}{2}\|\nabla u\|_2^2+\frac{q^2}{16\pi}\left(4\pi\int_{\RT}\phi_u u^2dx-a^2\|\Delta \phi_u\|_2^2\right)
				-\frac{q^2a^2}{16\pi}\|\Delta \phi_u\|_2^2-\frac{3q^2}{2}\int_{\RT}\phi_u u^2dx\\
				&~~~~+\frac{3}{p}\left(\|\nabla u\|_2^2+q^2\int_{\RT}\phi_u u^2dx\right)\\
				&=\frac{6-p}{2p}\|\nabla u\|_2^2+\frac{q^2(12-5p)}{4p}\int_{\RT}\phi_u u^2dx-\frac{q^2a^2}{8\pi}\|\Delta \phi_u\|_2^2.
		\end{aligned}
	\end{equation*}
So, in the case $p\ge6$, 
\[
0=\frac{6-p}{2p}\|\nabla u\|_2^2+\frac{q^2(12-5p)}{4p}\int_{\RT}\phi_u u^2dx-\frac{q^2a^2}{8\pi}\|\Delta \phi_u\|_2^2<0,
\]
which yields a contradiction. \\
On the other hand, by using Nehari identities,  if $p<12/7$, one has
    \begin{equation*}
		\begin{aligned}
				0&=\frac{6-p}{2p}\|\nabla u\|_2^2+\frac{q^2(12-5p)}{4p}\int_{\RT}\phi_u u^2dx-\frac{q^2a^2}{8\pi}\|\Delta \phi_u\|_2^2\\
				&=\frac{6-p}{2p}\|\nabla u\|_2^2+\frac{q^2(12-5p)}{4p}\int_{\RT}\phi_u u^2dx-\frac{q^2}{8\pi}\left(4\pi\int_{\RT}\phi_u u^2dx-\|\nabla \phi_u\|_2^2\right)\\
                &\ge q^2\left(\frac{12-5p}{4p}-\frac{1}{2}\right)\int_{\RT}\phi_u u^2dx>0,
		\end{aligned}
	\end{equation*}
obtaining again a contradiction.
\end{proof}

Now we can pass to prove the nonexistence result in the radial case.

\begin{proof}[The proof of Theorem \ref{Thm4-bis}]
Let $(v,\phi_v)\in \E_r\times \A_r$ be a nontrivial solution of system \eqref{eq1.1}. Then, for the case $p\le2$, it follows from Lemma \ref{lem6} and Remark \ref{rem1} that
    \begin{equation*}
        \begin{aligned}
            0&=\|\nabla v\|_2^2+q^2\int_{\RT}\phi_v v^2dx-\|v\|_p^p\\
            &=\|\nabla v\|_2^2+q^2\int_{\RT}\phi_v v^2dx\\
            &~~~~+\frac{p}{3}\left\{-\frac{1}{2}\|\nabla v\|_2^2-\frac{q^2}{4a}\int_{\RT}\int_{\RT}\left[\frac{5(1-e^{-\frac{|x-y|}{a}})}{|x-y|/a}+e^{-\frac{|x-y|}{a}}\right]v^2(x)v^2(y)dxdy\right\}\\
            &=\left(1-\frac{p}{6}\right)\|\nabla v\|_2^2+q^2\int_{\RT}\int_{\RT}\left[\left(\frac{1}{a}-\frac{5p}{12a}\right)\frac{(1-e^{-\frac{|x-y|}{a}})}{|x-y|/a}-\frac{p}{12a}e^{-\frac{|x-y|}{a}}\right]v^2(x)v^2(y)dxdy\\
            &\ge\frac{2}{3}\|\nabla v\|_2^2+\frac{q^2}{6a}\int_{\RT}\int_{\RT}\left[\frac{1-e^{-\frac{|x-y|}{a}}}{|x-y|/a}-e^{-\frac{|x-y|}{a}}\right]v^2(x)v^2(y)dxdy>0,
        \end{aligned}
    \end{equation*}
since the function in the parenthesis is positive. This establishes the desired result.
\end{proof}

\section{The behaviour as $a\to 0$ in the radial case}\label{se5}

In this section we focus on proving Theorem \ref{Thm1}, concerning the asymptotic behaviour of nontrivial weak solutions to system \eqref{eq1.1} as $a\to 0$, in the radial setting.

Following  Lemma \ref{lem6} and Remark \ref{rem1},  we define the Nehari-Pohozaev set
\begin{equation}\label{M}
    \mathcal{M}_{a,r}:=\Big\{u\in \E_r\setminus\{0\}: J_a(u)=0 \Big\},
\end{equation}
where
\begin{equation}\label{eq7.11}
	\begin{split}
		J_a(u):=&~2I_a'(u)[u]+P_a(u)\\
		=&~\frac32\|\nabla u\|_2^2-\frac{2p-3}{p}\|u\|_p^p +\frac{3q^2}{4}\int_{\R^3}\int_{\R^3}\frac{1-e^{-\frac{|x-y|}{a}}-\frac{|x-y|}{3a}e^{-\frac{|x-y|}{a}}}{|x-y|}u^2(x)u^2(y)dxdy.
	\end{split}
\end{equation}
Now, we present a key inequality.
\begin{lem}\label{lem7}
	Assume that $p\in(4,6)$. For each $t>0$, one has
	\begin{equation}\label{eq7.12}
		I_a(u)-I_a(t^2u(tx))\ge \frac{1-t^3}{3}J_a(u), ~~\forall u\in\E_r.
	\end{equation}
\end{lem}

\begin{proof}
	Noting that
	\begin{equation}\label{eq7.14}
		\begin{aligned}
			I_a(t^2u(tx))=\frac{t^3}{2}\|\nabla u\|_2^2+\frac{q^2t^3}{4}\int_{\R^3}\int_{\R^3}\frac{1-e^{-\frac{|x-y|}{at}}}{|x-y|}u^2(x)u^2(y)dxdy-\frac{t^{2p-3}}{p}\|u\|^p_p,
		\end{aligned}
	\end{equation}
	it thus follows from \eqref{eq7.11} that
		\begin{equation}\label{eq7.13}
			\begin{aligned}
		&I_a(u)-I_a(t^2u(tx))\\
		&=\frac{1-t^3}{2}\|\nabla u\|_2^2+\frac{t^{2p-3}-1}{p}\|u\|^p_p+\frac{q^2}{4}\int_{\R^3}\int_{\R^3}\frac{1-e^{-\frac{|x-y|}{a}}-t^3(1-e^{-\frac{|x-y|}{at}})}{|x-y|}u^2(x)u^2(y)dxdy\\
		&= \frac{1-t^3}{3}J_a(u)+\frac{1}{p}\|u\|_p^p\left[\frac{(1-t^3)(2p-3)}{3}+(t^{2p-3}-1)\right]\\
		&~~~~-\frac{q^2}{4}\int_{\R^3}\int_{\R^3}\frac{(1-t^3)\left(1-e^{-\frac{|x-y|}{a}}-\frac{|x-y|}{3a}e^{-\frac{|x-y|}{a}}\right)}{|x-y|}u^2(x)u^2(y)dxdy\\
		&~~~~+\frac{q^2}{4}\int_{\R^3}\int_{\R^3}\frac{1-e^{-\frac{|x-y|}{a}}-t^3(1-e^{-\frac{|x-y|}{at}})}{|x-y|}u^2(x)u^2(y)dxdy\\
		&= \frac{1-t^3}{3}J_a(u)+\frac{1}{p}\|u\|_p^p\left[\frac{(1-t^3)(2p-3)}{3}+(t^{2p-3}-1)\right]\\
		&~~~~+\frac{q^2}{4}\int_{\R^3}\int_{\R^3}\frac{t^3\left(e^{-\frac{|x-y|}{at}}-e^{-\frac{|x-y|}{a}}\right)+(1-t^3)\frac{|x-y|}{3a}e^{-\frac{|x-y|}{a}}}{|x-y|}u^2(x)u^2(y)dxdy.
			\end{aligned}
	\end{equation}
	Let 
	$g(t):=\frac{(1-t^3)(2p-3)}{3}+t^{2p-3}-1$ for any $t>0$. A simple calculation shows that
	\begin{equation*}\label{eq7.17}
		g(t)\ge g(1)=0,
	\end{equation*}
	which, combined with \eqref{eq3.24} and \eqref{eq7.13}, implies that 
	$$I_a(u)-I_a(t^2u(tx))\ge \frac{1-t^3}{3}J_a(u).$$
	
\end{proof}

The following corollary follows directly from Lemma~\ref{lem7}.
\begin{cor}\label{cor2}
	For $p \in (4,6)$ and any $u \in \mathcal{M}_{a,r}$, it holds that
	$$I_a(u)=\max_{t>0}I_a(t^2u(tx)).$$
\end{cor}

We proceed to establish a preliminary result that will play a key role in the analysis leading to Theorem \ref{Thm1}.
\begin{lem}\label{lem8}
	Let $p\in(4,6)$. Then for any $u\in \E_r\setminus\{0\}$, there exists a unique $t_u>0$ such that $t_u^2u(t_u\cdot)\in \mathcal{M}_{a,r}$.
\end{lem}

\begin{proof}
	Fix $u\in\E_r\setminus\{0\}$ and define $\zeta(t):=I_a(t^2u(tx))$ on $(0,+\infty)$. Clearly, it follows from \eqref{eq7.11} and \eqref{eq7.14} that
	\begin{equation}\label{eq7.23}
		\begin{aligned}
		\zeta'(t)=0
		\Leftrightarrow &\ \frac{3t^2}{2}\|\nabla u\|_2^2+\frac{3q^2t^2}{4}\int_{\R^3}\int_{\R^3}\frac{1-e^{-\frac{|x-y|}{at}}}{|x-y|}u^2(x)u^2(y)dxdy\\
		&\ -\frac{q^2t}{4a}\int_{\R^3}\int_{\R^3}e^{-\frac{|x-y|}{at}}u^2(x)u^2(y)dxdy
		-\frac{(2p-3)t^{2p-4}}{p}\|u\|^p_p=0\\
		\Leftrightarrow &\  \frac{1}{t}J_a(t^2u(tx))=0 \\
		\Leftrightarrow &\ t^2u(tx)\in \mathcal{M}_{a,r}.
		\end{aligned}
	\end{equation}
	One easily checks that $\lim\limits_{t\to 0^+}\zeta(t)=0$, $\zeta(t)>0$ for sufficiently small $t>0$, and $\zeta(t)<0$ for large $t$. Therefore, there exists $t_u>0$ such that $\zeta(t_u)=\max\limits_{t>0}\zeta(t)$, which means that 
	$$\zeta'(t_u)=0 ~~\text{ and }~~ t_u^2u(t_ux)\in \mathcal{M}_{a,r}.$$
	
	Finally, we claim that $t_u$ is unique. Indeed, for any given $u\in\E_r\setminus\{0\}$, suppose that $t_1,t_2>0$ satisfy  $\zeta'(t_1)=\zeta'(t_2)=0$. 
	Then we deduce from \eqref{eq7.23} that
	$$J_a(t_1^2u(t_1x))=J_a(t_2^2u(t_2x))=0.$$
	Combining this with \eqref{eq7.12}, one has
	$$I_a(t_1^2u(t_1x))\ge I_a(t_2^2u(t_2x)) ~~\text{ and }~~ I_a(t_2^2u(t_2x))\ge I_a(t_1^2u(t_1x)),$$
	which implies that $t_1=t_2$. 
	
	Therefore, we complete the proof.
	 
\end{proof}

As an immediate consequence of Proposition~\ref{prop1}, Corollary~\ref{cor2} and Lemma~\ref{lem8}, we obtain the following minimax property.
\begin{cor}\label{lem9}
	Let $p\in(4,6)$, then it holds
	$$0<c_a=\inf_{u\in \mathcal{M}_{a,r}}I_a(u)=\inf_{u\in \E_r\setminus\{0\} }\max_{t>0}I_a(t^2u(tx)).$$
\end{cor}

Next, we start to prove Theorem \ref{Thm1}.
\begin{proof}[The proof of Theorem \ref{Thm1}]
Let us denote by $c_0$ the corresponding ground state energy of system~\eqref{SP} for $p \in (4,6)$. Arguing as above we can show that 
    $$0<c_0=\inf_{v\in \mathcal{M}_{0,r}}I_0(v)=\inf_{v\in E_r\setminus\{0\} }\max_{t>0}I_0(t^2v(tx)),$$
        where $\M_{0,r}$ is the Nehari-Pohozaev set related to $I_0$. According to \cite[Theorem 1.1]{IR12}, there exists $u \in E_r\setminus\{0\}$ a ground state of system~\eqref{SP} for $p \in (4,6)$. 
Since $u\in E_r\hookrightarrow\E_r$, it follows from Lemma \ref{lem8} that there exists a unique value $t_u>0$ such that $t_u^2 u(t_u \cdot)\in \mathcal{M}_{a,r}$. 
Therefore	
\begin{equation}\label{eq7.18}
	c_0=I_0(u)\ge I_0(t_u^2 u(t_u x))\ge I_a(t_u^2 u(t_u x))\ge c_a.
\end{equation}

From now on, we assume that $\{(u_a,\phi_a)\}_{a>0}$ in $\E_r\times\mathcal{A}_r$ is the family of ground state solutions of system \eqref{eq1.1}, where we are using the notation $\phi_a:=\phi_{u_a}$. Then, by  Corollary \ref{lem9}, we have  
$$J_a(u_a)=0 ~~\text{ and }~~ I_a(u_a)=c_a.$$
Thus, it yields
\begin{equation}\label{eq3.26}
	\begin{aligned}
		c_a
		&=I_a(u_a)-\frac{1}{2p-3}J_a(u_a)\\
		&=\frac{p-3}{2p-3}\|\nabla u_a\|_2^2+\frac{q^2(p-3)}{2(2p-3)}\int_{\R^3}\int_{\R^3}\K_a(x-y)u_a^2(x)u_a^2(y)dxdy\\
		&~~~~+\frac{q^2}{4a(2p-3)}\int_{\R^3}\int_{\R^3}e^{-{\frac{|x-y|}{a}}}u_a^2(x)u_a^2(y)dxdy,
	\end{aligned}
\end{equation}
and so \eqref{eq7.18} gives that the family $\{u_a\}_{a>0}$ is bounded in $D_r^{1,2}(\R^3)$. Consequently, there exists $u_0\in D_r^{1,2}(\R^3)$ such that, up to subsequences, as $a\to 0$,
$$u_a\weakto u_0 \text{ weakly in } D_r^{1,2}(\R^3),$$
and
\begin{equation}\label{eq7.30}
	u_a\to u_0 \text{ strongly in } L_{\rm loc}^\tau(\R^3), \forall \tau\in[1,6).
\end{equation}
Furthermore, owing to the fact that $\int_{\R^3}\int_{\R^3}\K_a(x-y)u_a^2(x)u_a^2(y)dxdy\le c_a\le c_0$ obtained in \eqref{eq3.26}, it follows from Lemma \ref{thmE} that 
$$\|\nabla \phi_a\|_2^2 \le \|\nabla \phi_a\|_2^2 + a^2 \|\Delta \phi_a\|_2^2=4\pi\int_{\R^3}\int_{\R^3}\K_a(x-y)u_a^2(x)u_a^2(y)dxdy \le C.$$ 
Then also $\{\phi_a\}_{a>0}$ is bounded in $D_r^{1,2}(\R^3)$ and there exists $\phi_0\in D_r^{1,2}(\RT)$ such that
\begin{equation}\label{eq7.26}
	\phi_a\weakto \phi_0 \text{ weakly in } D_r^{1,2}(\RT), \text{ as } a\to 0.
\end{equation}
Passing to the limit as $a\to 0$ in the identity
$$\int_{\RT}\nabla\phi_a\nabla\eta dx+a^2\int_{\RT}\Delta\phi_a\Delta\eta dx=\int_{\RT}4\pi u_a^2\eta dx,~~ \eta\in C_c^\infty(\R^3),$$
and using that 
$$\left|a^2\int_{\RT}\Delta\phi_a\Delta\eta dx\right|\le a\|a\Delta\phi_a\|_2\|\Delta\eta\|_2\le aC\to 0,$$
we thus get
\begin{equation*}\label{eq7.27}
	\int_{\RT}\nabla\phi_0 \cdot \nabla\eta dx=\int_{\RT}4\pi u_0^2\eta dx,~~ \eta\in C_c^\infty(\R^3).
\end{equation*} 

Set $\varphi\in C_c^\infty(\R^3)$ with $\Omega:={\rm supp}(\varphi)$. We know that
\begin{equation}\label{eq2.4}
	\int_{\Omega}\nabla u_a \cdot\nabla \varphi dx+q^2 \int_{\Omega}\phi_{a}u_a\varphi dx=\int_{\Omega}|u_a|^{p-2}u_a\varphi dx.
\end{equation}
Applying H\"older inequality and \eqref{eq7.30}, we get, as $a\to0$,
$$\int_{\Omega}\nabla u_a \cdot\nabla \varphi dx\to\int_{\Omega}\nabla u_0 \cdot\nabla \varphi dx
~~\text{ and }~~
\int_{\Omega}|u_a|^{p-2}u_a\varphi dx\to \int_{\Omega}|u_0|^{p-2}u_0\varphi dx.
$$
Furthermore, by  \eqref{eq7.30} and \eqref{eq7.26} again, we also have that
\begin{equation*}
	\begin{aligned}
		&\left|\int_{\Omega}\phi_{a}u_a\varphi dx-\int_{\Omega}\phi_{0}u_0\varphi dx\right|\\
		&\le \int_{\Omega}|\phi_{a}-\phi_{0}||u_a||\varphi| dx+\int_{\Omega}|\phi_{0}||u_a-u_0||\varphi| dx\\
		&\le\left(\int_{\Omega}|\phi_{a}-\phi_{0}|^2dx\right)^{\frac12}\left(\int_{\Omega}|u_a|^6 dx\right)^{\frac16}\left(\int_{\Omega}|\varphi|^3 dx\right)^{\frac13}\\
		&~~~~+\left(\int_{\Omega}|\phi_{0}|^6 dx\right)^{\frac16}\left(\int_{\Omega}|u_a-u_0|^3 dx\right)^{\frac13}\left(\int_{\Omega}|\varphi|^2 dx\right)^{\frac12}\to0, \quad\text{ as } a\to 0.
	\end{aligned}
\end{equation*}
As a consequence, for any $\varphi\in C_c^{\infty}(\R^3)$, we deduce from \eqref{eq2.4} that, as $a\to 0$,
$$\int_{\Omega}\nabla u_0 \cdot\nabla \varphi dx+q^2 \int_{\Omega}\phi_{0}u_0\varphi dx=\int_{\Omega}|u_0|^{p-2}u_0\varphi dx.$$

The proof of Theorem \ref{Thm1} is thus concluded.

\end{proof}

\bigskip

{\bf Acknowledgments}
A.P. is member of INdAM-GNAMPA. A.P. is financed by European Union - Next Generation EU - PRIN 2022 PNRR ``P2022YFAJH Linear and Nonlinear PDE's: New directions and Applications",  by the Italian Ministry of University and Research under the Program Department of Excellence L. 232/2016 (Grant No. CUP D93C23000100001), and by  INdAM-GNAMPA Project {\em Metodi variazionali e
topologici per alcune equazioni di Schr\"odinger nonlineari} (CUP E53C23001670001). 
L.Y. is supported by the China Scholarship Council (Grant No. 202406660027). He is grateful to Prof. Erasmo Caponio (Politecnico di Bari) for his valuable comments and suggestions on this paper, and also thanks Ph.D. Yafei Li (Beihang University) for helpful discussions.

\smallskip

{\bf Data availability} Data sharing not applicable to this article as no datasets were generated
or analysed during the current study.

\smallskip

{\bf Conflict of interest} The authors declare that they have no conflict of interest.

\end{document}